\documentclass[11pt,a4paper,reqno]{amsart}

\usepackage[utf8]{inputenc}
\usepackage[T1]{fontenc}
\usepackage[british]{babel}
\usepackage{amsmath,amsthm,amssymb}
\usepackage{hyperref,float,caption,comment}
\hypersetup{
    colorlinks,
    linkcolor={red!60!black},
    citecolor={green!60!black},
    urlcolor={blue!60!black},
}
\usepackage{geometry}
\usepackage{enumitem,thm-restate}
\usepackage[abbrev, msc-links, nobysame]{amsrefs}
\usepackage{tikz,esvect}
\usetikzlibrary{arrows}
\tikzset{
e+ /.tip = {[sep=0pt]|[sep=3pt]_},
e+4 /.tip = {[sep=0pt]|[sep=5pt]_},
e+5 /.tip = {[sep=0pt]|[sep=5pt]_},
e+8 /.tip = {[sep=0pt]|[sep=8pt]_},
e+2 /.tip = {[sep=0pt]|[sep=2pt]_}
}
\tikzset{arc/.style = {->,> = stealth'}}
\tikzset{
	dot/.style = {circle, fill, minimum size=#1,
		inner sep=0pt, outer sep=0pt},
	dot/.default = 6pt
}
\usetikzlibrary{angles,quotes}
\usepackage[nameinlink, capitalise, noabbrev, sort]{cleveref}
\usepackage{subcaption}

\definecolor{CornflowerBlue}{rgb}{0.39, 0.58, 0.93}
\definecolor{LavenderMagenta}{rgb}{0.93, 0.51, 0.93}
\definecolor{PastelOrange}{rgb}{1.0, 0.7, 0.28}

\def\x{{\bf x}}
\def\y{{\bf y}}
\def\z{{\bf z}}
\def\a{{\bf a}}
\def\b{{\bf b}}
\def\c{{\bf c}}

\def\1{{\bf 1}}
\def\0{{\bf 0}}

\newtheorem{theorem}{Theorem}
\newtheorem{lemma}[theorem]{Lemma}
\newtheorem{corollary}[theorem]{Corollary}

\newtheorem{problem}[theorem]{Problem}

\theoremstyle{definition}

\Crefname{question}{Question}{Questions}

\title[]{A generalisation of Menger's theorem in bidirected graphs}

\author[E.\,Ghorbani]{Ebrahim Ghorbani}
\address{Ebrahim Ghorbani:
Hamburg University of Technology, Institute for Algorithms and Complexity, Hamburg, Germany}
\email{ebrahim.ghorbani@tuhh.de}
\author[J.K.\,Nickel]{Jana Katharina Nickel}
\author[F.\,Reich]{Florian Reich}
\address{Jana K. Nickel, Florian Reich: Universit{\"a}t Hamburg, Department of Mathematics, Bundesstra{\ss}e~55 (Geomatikum), 20146~Hamburg, Germany}
\email{\{jana.nickel-2,  florian.reich\}@uni-hamburg.de}

\begin{document}

\begin{abstract}
    Menger's theorem - the maximum number of vertex-disjoint $X$--$Y$~paths is equal to the minimum size of an $X$--$Y$~separator - is generally not true in bidirected graphs.
    We prove that Menger's theorem holds true if we take the nontrivial $X$--$X$~paths and the nontrivial $Y$--$Y$~paths into account.
\end{abstract}

\maketitle 

\section{Introduction}
A \emph{bidirected graph} is a multigraph $G$ together with a \emph{sign function} that maps a sign $\alpha\in\{+,-\}$ to each endpoint of every edge of $G$.
The purpose of the sign function is to restrict the paths in bidirected graphs:
Every two consecutive edges $uv$ and $vw$ of a path have opposite signs at their common vertex $v$.
In this way, bidirected graphs generalise directed graphs, for which we map all heads to the same sign and all tails to the opposite sign.

Bidirected graphs have been introduced independently by Kotzig~\cites{Kotzig1959,Kotzig1959b,Kotzig1960} and by Edmonds and Johnson~\cite{Edmonds2003}, and have recently attracted increasing interest from the graph theoretic community~\cites{wiederrecht2020thesis,bowler2023mengertheorembidirected,bowler2024hitting,nickel2025disjoint,abrishami2025structure,bowler2025connectoids,kita2017bidirected}.
This is particularly due to the one-to-one correspondence of bidirected graphs and graphs with a perfect matching, which Wiederrecht~\cite{wiederrecht2020thesis} suggested as a tool to solve Norine's Matching Grid conjecture~\cite{norine2005matching}.
The structure of bidirected graphs is more complex than the structure of directed graphs and, despite recent advances, not yet well understood.

Menger's famous theorem~\cite{Menger1927} -- the maximum number of of vertex disjoint paths between two sets $X$ and $Y$ of vertices is equal to the minimum size of an $X$--$Y$~separator -- is generally not true in bidirected graphs~\cite{bowler2023mengertheorembidirected}.
Bowler, Gut, Jacobs, the first author and the third author introduced the following sufficient condition on $X$ and $Y$ for Menger's theorem:
\begin{theorem}[\cite{bowler2023mengertheorembidirected}*{Theorem 1.2}]\label{thm:clean_menger}
    Let $B$ be a bidirected graph and let $X$ and $Y$ be sets of vertices of $B$ such that there is either no nontrivial $X$--$X$~path or no nontrivial $Y$--$Y$~path.
    Then the maximum number of vertex-disjoint $X$--$Y$~paths is equal to the minimum size of a set $S$ of vertices such that there is no $X$--$Y$~path in $B - S$.
\end{theorem}

In this paper, we generalize~\cref{thm:clean_menger} to general sets $X$ and $Y$ under taking into account the nontrivial $X$--$X$ and $Y$--$Y$~paths, where paths consisting of a single vertex are
called trivial.
We call the vertex-disjoint union of a nontrivial $X$--$X$~path and a nontrivial $Y$--$Y$~path an \emph{$X$--$Y$~turnaround}.
An \emph{$X$--$Y$~link} is either an $X$--$Y$~path or an $X$--$Y$~turnaround.
We show:
\begin{theorem}\label{thm:main}
    Let $B$ be a bidirected graph, and let $X$ and $Y$ be sets of vertices of $B$.
    Then the maximum number of vertex-disjoint $X$--$Y$~links, where $X$--$Y$~turnarounds are counted twice, is greater than or equal to the minimum size of a set $S$ of vertices such that there is no $X$--$Y$~link in $B - S$.
\end{theorem}

As a special case of~\cref{thm:main} we obtain a strengthening of~\cref{thm:clean_menger}, where we allow $X$--$X$~paths and $Y$--$Y$~paths that intersect pairwise:
\begin{corollary}\label{thm:clean_menger_strengthening}
    Let $B$ be a bidirected graph, and let $X$ and $Y$ be sets of vertices of $B$ such that there is no $X$--$Y$~turnaround.
    Then the maximum number of vertex-disjoint $X$--$Y$~paths is equal to the minimum size of a set $S$ of vertices such that there is no $X$--$Y$~path in $B - S$.
\end{corollary}

Furthermore, we will deduce the following variant for $X$-paths, which was first proven by the second author:

\begin{corollary}[\cite{nickel2025disjoint}*{Theorem 4.1}]\label{cor:ep}
    Let $B$ be a bidirected graph, and let $X$ be a set of vertices of $B$.
    Then twice the maximum number of vertex-disjoint $X$-paths bounds the minimum size of a set $S$ of vertices such that there is no $X$-path in $B - S$.
\end{corollary}

Our proof of~\cref{thm:main} follows the approach of the proof of~\cite{bowler2024hitting}*{Theorem 2.3} building on linear programming.
In comparison to the known proofs of~\cref{thm:clean_menger} \cites{bowler2023mengertheorembidirected,abrishami2025structure} and \cref{cor:ep}~\cite{nickel2025disjoint}, our proof of~\cref{thm:main} (and the deduction of \cref{thm:clean_menger_strengthening,cor:ep}) is significantly shorter.

We remark that \cref{thm:main} is best possible, that is \cref{thm:main} is sharp and counting $X$--$Y$~turnarounds twice is indeed necessary: Let $B$ be some bidirected graph on the disjoint union of two~$K_3$, where $X$ is the vertex set of one of them and $Y$ is the vertex set of the other one (see~\cref{fig:example-graph1}).
Then the maximum number of vertex-disjoint $X$--$Y$~links and the minimum size of a set~$S$ of vertices such that there is no $X$--$Y$~link in $B - S$ are both two.
Furthermore, we note that the values in~\cref{thm:main} are generally not equal: Let $B'$ be the bidirected graph obtained from $B$ by removing an edge incident with $X$ (see~\cref{fig:example-graph2}), then the maximum number of vertex-disjoint $X$--$Y$~links stays two, but there is a one-element set~$S$ such that $B- S$ does not contain an $X$--$Y$~link.

Finally, we ask for an edge-variant of~\cref{thm:main}:
\begin{problem}\label{problem}
    Let $B$ be a bidirected graph and let $s$ and $t$ be vertices of~$B$.
    Is the maximum number of edge-disjoint $s$--$t$~edge-links, where $s$--$t$~edge-turnarounds are counted twice,  greater than or equal to the minimum size of a set $S$ of edges such that there is no $s$--$t$~edge-link in $B - S$?
\end{problem}
\noindent
Here, an \emph{$s$--$t$~edge-turnaround} is the edge-disjoint union of an $s$--$s$~trail and a $t$--$t$~trail.
Similarly, an \emph{$s$--$t$~edge-link} is either an $s$--$t$~path or an $s$--$t$~edge-turnaround.
We remark that~\cref{problem} fails if we define $s$--$t$~edge-turnarounds
as unions of $s$--$s$ and $t$--$t$~trails such that the internal vertices of each trail are distinct (see~\cite{bowler2023mengertheorembidirected}*{Figure 3}).

\begin{figure}
\centering
\begin{subfigure}[b]{0.44\textwidth}
\centering
\begin{tikzpicture}
\coordinate (a) at (0,0);
\coordinate (b) at (1,0.577);
\coordinate (c) at (1,-0.577);

\coordinate (d) at (4,0);
\coordinate (e) at (3,0.577);
\coordinate (f) at (3,-0.577);
\foreach \i in {a,b,c}{
    \draw[PastelOrange, fill=PastelOrange] (\i) circle (5pt);
}
\foreach \i in {d,e,f}{
    \draw[CornflowerBlue, fill=CornflowerBlue] (\i) circle (5pt);
}
\foreach \i in {a,b,c,d,e,f}{
    \node[dot] at (\i) [] {};
}
\foreach \i/\j in {a/b,b/c,c/a,b/a,c/b,a/c,d/e,e/d,e/f,f/e,d/f,f/d}{
    \draw[-e+8,thick] (\i) to (\j);
}
\node[PastelOrange] at (-0.5,0) [] {$X$};
\node[CornflowerBlue] at (4.5,0) [] {$Y$};
\end{tikzpicture}
\caption{The bidirected graph $B$ has the property that the maximum number of disjoint $X$--$Y$~links is two and there exists no vertex $v$ for which there is no $X$--$Y$~link in $B - v$.}
\label{fig:example-graph1}
\end{subfigure}
\hfill
\begin{subfigure}[b]{0.44\textwidth}
\centering
\begin{tikzpicture}
\coordinate (a) at (0,0);
\coordinate (b) at (1,0.577);
\coordinate (c) at (1,-0.577);

\coordinate (d) at (4,0);
\coordinate (e) at (3,0.577);
\coordinate (f) at (3,-0.577);
\foreach \i in {a,b,c}{
    \draw[PastelOrange, fill=PastelOrange] (\i) circle (5pt);
}
\foreach \i in {d,e,f}{
    \draw[CornflowerBlue, fill=CornflowerBlue] (\i) circle (5pt);
}
\foreach \i in {a,b,c,d,e,f}{
    \node[dot] at (\i) [] {};
}
\foreach \i/\j in {a/b,c/a,b/a,a/c,d/e,e/d,e/f,f/e,d/f,f/d}{
    \draw[-e+8,thick] (\i) to (\j);
}
\node[PastelOrange] at (-0.5,0) [] {$X$};
\node[CornflowerBlue] at (4.5,0) [] {$Y$};
\node[] at (0,-0.4) [] {$a$};
\end{tikzpicture}
\caption{The bidirected graph $B'$ has the property that the maximum number of disjoint $X$--$Y$~links is two and there is no $X$--$Y$~link in $B - a$.}
\label{fig:example-graph2}
\end{subfigure}
\caption{Two bidirected graphs in which every $X$--$Y$~link is an $X$--$Y$~turnaround. The perpendicular bar at each half-edge indicates that it has sign $+$.}
\label{fig:example-graph}
\end{figure}

\section{From vertices to sets of vertices}
For a formal introduction to bidirected graphs, we refer the reader to~\cite{bowler2023mengertheorembidirected}.
Let $s$ and $t$ be vertices of a bidirected graph~$B$.
We call an $s$--$s$~trail whose vertices are distinct except the startvertex and the endvertex an \emph{$s$--$s$~almost path}.
An \emph{$s$--$t$~turnaround} is the vertex-disjoint union of an $s$--$s$~almost path and a $t$--$t$~almost path.
Similarly, an \emph{$s$--$t$~link} is either an $s$--$t$~turnaround or an $s$--$t$~path.
Two $s$--$t$~links are \emph{internally vertex-disjoint} if they intersect only in $s$ and~$t$.
We will prove the following variant of~\cref{thm:main}:
\begin{theorem}\label{thm:main_vtx}
    Let $B$ be a bidirected graph, and let $s$ and $t$ be distinct vertices of~$B$.
    Then the maximum number of internally vertex-disjoint $s$--$t$~links, where we count $s$--$t$~turnarounds twice, bounds the minimum size of a set $S \subseteq V(B) \setminus \{s,t\}$ such that there is no $s$--$t$~link in $B - S$.
\end{theorem}

\begin{proof}[Proof of \cref{thm:main} from \cref{thm:main_vtx}]
Let $\hat{B}$ denote the auxiliary graph obtained from $B$ by adding an auxiliary vertex $s$ together with, for every element $x \in X$ and each sign $\alpha \in \{+,-\}$, an edge $x s$ with sign $\alpha$ at~$x$ and sign $-$ at $s$.
Similarly, we add an auxiliary vertex $t$ with edges towards the vertices of~$Y$.
Then the $s$--$t$~links in $\hat{B}$ correspond one-to-one to the $X$--$Y$~links in~$B$, which completes the proof.
\end{proof}

\begin{proof}[Proof of \cref{cor:ep} from \cref{thm:main}]
    Let $\hat{B}$ be the disjoint union of two copies of $B$, where $X'$ and $X''$ are the sets corresponding to $X$.
    Note that there are no $X'$--$X''$~paths in $B$.
    By~\cref{thm:main}, twice the maximum number of vertex-disjoint $X'$--$X''$~turnarounds bounds the minimum size of a set $\hat{S}$ of vertices in $\hat{B}$ such that there is no $X'$--$X''$~turnaround in $\hat{B} - \hat{S}$.
    
    The maximum number of $X'$--$X''$~turnarounds in $\hat{B}$ is equal to the maximum number of $X$-paths in $B$.
    Similarly, the minimum size of a set $\hat{S}$ of vertices of $\hat{B}$ such that there is no $X'$--$X''$~turnaround in $\hat{B} - \hat{S}$ is equal to the minimum size of a set $S$ of vertices in $B$ such that there is no $X$-path in $B - S$. This proves \cref{cor:ep}.
\end{proof}

\section{Proof of~\cref{thm:main_vtx}}

While the total unimodularity of incidence matrices of directed graphs relies on the characteristic property that each column has exactly one $+1$ and one $-1$ entry, this characteristic does not necessarily hold for bidirected graphs, whose incidence matrices may thus not be totally unimodular. 
Despite this, the following more general notions allow us to extend the integrality of optimal solutions of the linear programs considered for directed graphs to the context of bidirected graphs.

A rational matrix $A$ is called {\em $k$-regular}, for a positive integer $k$, if for every non-singular submatrix $R$ of $A$, $kR^{-1}$ is integral.
If $A$ is $k$-regular, then the matrix $\frac{1}{k} A$ is obviously $1$-regular.

\begin{lemma}[\cite{APPA2004k-regular}*{Theorem~23}]\label{lem:BiDirIncidence2Regular}
Let $A=(a_{ij})_{m \times n}$ be an integral matrix satisfying  
$\sum_{i=1}^{m} |a_{ij}| \leq 2$ for each $j = 1, \dots, n.$
Then $A$ is $2$-regular.
\end{lemma}

Further, a polyhedron $P=\{\x:  A\x\le\b, \x\ge\0\}$ is called \emph{integral} if
$\max\{\c^\intercal \x: \x \in P\}$ is attained by an integral vector for each $\c$ for which the maximum is finite. 

\begin{lemma}[\cite{APPA2004k-regular}*{Theorem 16}]\label{lem:integral-Ax<a}
    A rational $m\times n$ matrix~$A$ is $1$-regular if and only if for any integral vector $\b$ of length $m$, the polyhedron $\{\x:  A\x\le\b, \x\ge\0\}$ is integral.
\end{lemma}

The \emph{incidence matrix} of a bidirected graph $(G, \sigma)$ is defined by means of the signs of the half-edges, where the sign of the entries $+1$ or $-1$ corresponding to a vertex $v$ and an edge $e$ is equal to $\sigma(v,e)$.

\begin{lemma}\label{lem:Integrality}
Let $M$ be the incidence matrix of a bidirected graph, and let $\a$ and~$\c$ be some integral vectors. Then the following polyhedra are integral:
\begin{enumerate}
\item[\rm(i)] $\left\{ \x : \frac12M\x = \a, \0 \leq \x \leq \c\right\}$,

\medskip

\item[\rm(ii)] $\left\{\begin{bmatrix}
			\z\\
			\y
		\end{bmatrix} : \frac12M^\intercal \z + \y \geq \a, \z \geq \0\right\}$.
\end{enumerate}
\end{lemma}
\begin{proof}
We define
\begin{equation*}
		A:=
		\begin{bmatrix}
			I \\
			\frac12M \\
			-\frac12M
		\end{bmatrix}\quad \text{and}\quad
 		\b:=
		\begin{bmatrix}
			\c\\
			\a\\
			\a
		\end{bmatrix}.
\end{equation*}
Then the polyhedron $\{\x:  A\x\le\b, \x\ge\0\}$ is the same as the one given in (i).
	By~\cref{lem:BiDirIncidence2Regular}, the matrix $M$ is $2$-regular, and hence, $\frac{1}{2}M$ is $1$-regular.
	It follows that the matrix
	$\begin{bmatrix}
		\frac12M \\
		-\frac12M
	\end{bmatrix}$
	is also $1$-regular because any of its non-singular submatrices can be viewed as a submatrix of $\frac{1}{2}M$ with some rows possibly negated.
	Stacking an identity matrix on a $1$-regular matrix again produces a $1$-regular matrix (see \cite{APPA2004k-regular}*{Lemma~7}), so $A$ is again $1$-regular. Thus, applying \cref{lem:integral-Ax<a}, we see that the polyhedron in (i) is integral.
    
   Integrality of the polyhedron (ii) follows in a similar way, when noting that we may write $\y$ as $\y=\y'- \y''$ with $\y', \y''\ge\0$, and using the fact that the matrix $A=\begin{bmatrix}
			-\frac12M^\intercal& -I&I
\end{bmatrix}$ is $1$-regular.
\end{proof}

\begin{proof}[\bf Proof of \cref{thm:main_vtx}]
Without loss of generality, we can assume that all half-edges incident with $s$ have a negative sign and all those incident with $t$ have a positive sign.

We replace each vertex $v \in V(B)\setminus\{s,t\}$ by two vertices $v^+$ and~$v^-$, where we make $v^+$ incident with all edges that are incident with $v$ with sign $+$ and similarly make $v^-$ incident with all edges incident with $v$ with sign $-$.
Furthermore, we connect $v^-$ and $v^+$ by an edge with sign $-$ at $v^+$ and sign $+$ at $v^-$.
Finally, we join $t$ to $s$ by a new edge~$f$ with sign $+$ at~$s$ and sign $-$ at~$t$.
Let $B'$ denote the resulting bidirected graph. 
We record the following observations:

\begin{itemize}
\item There is a one-to-one correspondence between the $s$--$t$~paths of $B$ and those of $B'$.
\item There is a one-to-one correspondence between the $s$--$t$~turnarounds of $B$ and those of~$B'$.
\item Two $s$--$t$~links in $B$ are internally vertex-disjoint if and only if the corresponding links in $B'$ are edge-disjoint.
\item The incidence matrix of $B'$ is $\bigl[M~\a\bigr]$, where $\a$ corresponds to $f$ and $M$ is the incidence matrix of $B'-f$.
\end{itemize}

Let us consider the primal linear program
    
	\begin{equation}\tag{P}\label{eq:LP}
\begin{aligned}
	\text{ maximise } &x_f \\
	\text{ subject to } &\tfrac12M\x+\tfrac12\a x_f =\0, \\
	&\0\le\x \leq \1,\\
	&x_f\ge0.
\end{aligned}
	\end{equation}
	By \cref{lem:Integrality}, \eqref{eq:LP} has an integral optimal solution 	$\begin{bmatrix}
	\x^*\\x^*_f
	\end{bmatrix}$.
The optimal value $x^*_f$ of \eqref{eq:LP} gives the maximum number of edge-disjoint $s$--$t$~links, with $s$--$t$~turnarounds counted twice for the reason clarified below.
Consider the subgraph induced by this optimal solution, obtained by including edges with a value of $1$ in $\mathbf{x}^*$, and since $x^*_f$ could be greater than 1, including $x^*_f$ copies of the edge $f$. 
The resulting bidirected graph is \emph{balanced} in the sense that, at each vertex, the number of $+$-half edges and $-$-half edges coincide.
By virtue of the structure of $B'$, this gives a collection of $s$--$t$~links.  
Each $s$--$t$~turnaround needs to use two copies of $f$ at $s$ to fulfill the balance property.

	The dual of \eqref{eq:LP} has a variable $z_v$ for every  $v \in V(B')$ and  $y_e$ for every $e \in E(B')$:  
	\begin{equation}\tag{D}\label{eq:Dual}
	\begin{aligned}
		\text{ minimise } &\1^\intercal \y \\
		\text{ subject to } &\tfrac12M^\intercal\z+\y \ge\0, \\
		 &\tfrac12\a^\intercal\z\ge1, \\
		&\y \ge \0.
	\end{aligned}
\end{equation}
	By \cref{lem:Integrality}, \eqref{eq:Dual} has an integral optimal solution 
		$\begin{bmatrix}
		\y^*\\ \z^*
	\end{bmatrix}$.
 Furthermore, by the strong duality theorem of linear programming (see, for example, \cite{schrijver1998}*{Corollary~19.2b}), the optimal value of~\eqref{eq:Dual} equals that of~\eqref{eq:LP}, meaning that $x^*_f=\1^\intercal \y^*$.
Since our optimal solution must satisfy the constraints of~\eqref{eq:Dual}, we have
\begin{align}
\sigma(u,e)z^*_u+\sigma(v,e)z^*_v+2y^*_e &\ge0 \quad \text{ for all }~ e=uv \in E,\label{eq:sigma(u,e)z_u}\\
z^*_s-z^*_t&\ge2.\label{eq:z_s-z_t>2}
\end{align}
Define $F$ to be the set 
	$$F:=\bigl\{uv \in E(B'):  \sigma(u,e)z^*_u+\sigma(v,e)z^*_v<0\bigr\}.$$
From \eqref{eq:sigma(u,e)z_u}, it follows that $y^*_e\ge1$ for every $e\in F$, and thus,
	$$|F|\le \1^\intercal\y^*=x^*_f.$$ 
On the other hand, every $s$--$t$~link $P$ necessarily intersects~$F$.  
To justify this, we observe that since $P$ is balanced at its internal vertices, we have
	$$\sum_{e=uv\in E(P)} \bigl(\sigma(u,e)z^*_u+\sigma(v,e)z^*_v\bigr)=
	\begin{cases}
			z^*_t-z^*_s & \text{if $P$ is an $s$--$t$~path,}  \\  
			2(z^*_t-z^*_s) & \text{if $P$ is an $s$--$t$~turnaround.}
	\end{cases}$$
The above summation is negative by \eqref{eq:z_s-z_t>2}, which in turn ensures that the intersection $F \cap E(P)$ is non-empty.  
Consequently, $B - F$ contains no $s$--$t$~links, which completes the proof.
\end{proof}

\section*{Acknowledgements}
The third author gratefully acknowledges support by a doctoral scholarship of the Studienstiftung des deutschen Volkes.

\bibliographystyle{unsrtnat}
\bibliography{reference}

\end{document}